\documentclass{article}
\usepackage{amsmath}

\usepackage{amssymb}
\usepackage{epsfig}
\usepackage{graphicx}
\usepackage{tikz}
\usepackage{bm}
\usepackage{authblk}

\numberwithin{equation}{section}

\newtheorem{theorem}{Theorem}[section]
\newtheorem{lemma}[theorem]{Lemma}

\newtheorem{corollary}[theorem]{Corollary}
\newtheorem{proposition}[theorem]{Proposition}
\newtheorem{definition}[theorem]{Definition}

\newtheorem{remark}[theorem]{Remark}

\newenvironment{proof}{{\bf Proof.}}{}

\title{The fragment of elementary plane Euclidean geometry based on perpendicularity alone with complexity PSPACE-complete}

\author[1]{Tatyana Ivanova}
\author[2]{Tinko Tinchev}
\affil[1]{Institute of Mathematics and Informatics\authorcr
		Bulgarian Academy of Sciences\authorcr
		e-mail: tatyana.ivanova@math.bas.bg}
\affil[2]{Faculty of Mathematics and Informatics\authorcr
		Sofia University "St. Kliment Ohridski"\authorcr
		e-mail: tinko@fmi.uni-sofia.bg}
\date{}

\begin{document}

\maketitle

\bigskip
\begin{abstract}
	
A. Tarski uses in his system for the elementary geometry only the primitive concept of \textit{point}, and the two primitive relations \textit{betweenness} and \textit{equidistance}. Another approach is the relations to be on \textit{lines} instead of points. W. Schwabh\"{a}user and L. Szczerba showed that \textit{perpendicularity} together with the ternary relation of \textit{co-punctuality} are sufficient for dimension two, i.e. they may be used as a system of primitive relations for elementary plane Euclidean geometry. In this paper we give a complete axiomatization for the fragment of elementary plane Euclidean geometry based on perpendicularity alone. We show that this theory is not finitely axiomatizable, it is decidable and the complexity is PSPACE-complete. In contrast the complexity of elementary plane Euclidean geometry is exponential.

\end{abstract}	
	
\section{Introduction}

At the end of the 19th and the beginning of the 20th century investigations of axiomatic approach to geometry were done by Hilbert, Pieri, Veblen and Pasch (see \cite{Hilbert, Pieri, Veblen04, Veblen14, Pasch}). 

The first-order theory of Euclidean geometry was called elementary Euclidean geometry.

Pieri introduced the relation $\mathbf{\Delta}$ with the following definition: $\Delta (abc) := ||ab|| = ||ac||$. The meaning of this relation is that the points $a$, $b$ and $c$ form an isosceles triangle with base $|bc|$ (in the degenerated cases either $b = c$ or $a$ is the midpoint of $|bc|$). Pieri proved that $\Delta$ is sufficient as only primitive relation in terms of which the whole elementary $n$-dimensional Euclidean geometry for $n \geq 2$ can be developed.

Veblen (\cite{Veblen04}) considered two primitive relations - the ternary relation of \textit{betweenness} $\mathbf{B}$ and the quaternary relation of \textit{equidistance} $\bm{\equiv}$. If $a$, $b$, $c$ and $d$ are points, then $B(a,b,c)$ means that $b$ lies on the line segment joining $a$ and $c$; $ab\equiv cd$ means that the distance from $a$ to $b$ is the same as the distance from $c$ to $d$. Veblen proved that these primitive relations are sufficient for the elementary geometry.
He showed that $\equiv$ is undefinable in terms of $B$, although he believed that he had proved the contrary (\cite{Veblen14}). This mistake was noticed by Tarski (\cite{Tarski56}).

On the other hand, Pieri (\cite{Pieri}) showed that $B$ is first-order definable in terms of $\equiv$ for every dimension $n\geq 2$. Thus $\equiv$ can be taken as the only primitive for the elementary geometry.
	
In the 1920-70’s Tarski created an axiomatic system for the elementary $n$-dimensional Euclidean geometry for $n\geq 2$ and this axiomatic system was improved by Tarski and his students. The obtained theory is complete and decidable, which was proved by means of quantifier elimination. In connection with these results see \cite{Tarski51,Tarski59, Tarski67, Gupta,SchwabhauserSzmielewTarski, TarskiGivant}. Tarski's decision procedure had nonelementary complexity. Later decision procedures, having elementary complexity, were given by Monk, Solovay, Collins, Heintz, Roy, Solerno, Renegar, Basu, Pollack and others (\cite{Collins, Monk, Heintz, Renegar, Basu, BasuPollackRoy}). The complexity of these decision procedures is exponential. 

If in Tarski's system of geometry we replace the only axiom schema with a second-order axiom we obtain an axiomatic system for full (non-elementary) $n$-dimensional Euclidean geometry for $n\geq 2$ (see \cite{Tarski59, Tarski67}). If we remove the dimension axiom, then we get a complete axiomatization of the dimension-free Euclidean geometry (\cite{Scott}).

Beth and Tarski (\cite{BethTarski}) proved that no binary relation can be sufficient for the elementary $n$-dimensional geometry for every $n \geq 2$. Still, Beth and Tarski showed in \cite{BethTarski} that the ternary relation $\mathbf{E}$, where $E(abc)$ means that the points $a$, $b$ and $c$ form an equilateral triangle (or $a$, $b$ and $c$ coincide), can be taken as the only primitive for every dimension $n\geq 3$ but it is not sufficient for dimension $1$ or $2$.

Szczerba and Tarski (\cite{SzczerbaTarski65,SzczerbaTarski79}) gave an axiomatization of the affine fragment of the elementary plane Euclidean geometry - the fragment based on betweenness relation alone.
	
Tarski used in his system for the elementary geometry only the primitive concept of point, and the two primitive relations betweenness and equidistance - the same primitive relations as Veblen. Another approach is the relations to be on lines instead of points. Schwabh\"{a}user and Szczerba (\cite{SchwabhauserSzczerba}) studied relations on lines which can be taken as primitives for the elementary Euclidean geometry. They proved that the binary relation of \textit{perpendicularity} $\mathbf{O}$ is sufficient for all dimensions higher than three; for dimension three perpendicularity together with the binary relation of intersection of lines suffice.  Schwabh\"{a}user and Szczerba  considered also the ternary relation of \textit{co-punctuality} $\bm{\rho}$, meaning that three lines intersect at one point.
They showed that $O$ together with $\rho$ may be used as a system of primitive relations for elementary plane Euclidean geometry. The paper \cite{BalbianiTinchev} studies the first-order theory of the fragment of elementary plane Euclidean geometry based on the binary relations \textit{parallelism} ($\mathbf{P}$) and \textit{convergence} ($\mathbf{C}$) of lines. The predicates $P$ and $C$ are definable by $O$ and the converse is not true.

In this paper we give a complete axiomatization for the fragment of elementary plane Euclidean geometry based on perpendicularity alone. We show that this theory is not finitely axiomatizable, it is decidable and the complexity is PSPACE-complete. Moreover the theory is $\omega$-categorical and not categorical in every uncountable cardinality.

\section{A complete axiomatization based on perpendicularity alone}
We consider first-order language $\bm{\mathcal{L}}$ with predicate symbols $O$ (meaning perpendicularity) and $=$, without constants and functional symbols. We consider the theory $\mathbf{SAPP}$, containing the following formulas:

\vspace{5 mm}
\noindent
$\mathbf{\lambda_{1n}}:\ \forall y_1\ldots\forall y_n\forall s(O(s,y_1)\wedge\ldots \wedge O(s,y_n)\rightarrow \exists t(t\neq y_1\wedge\ldots\wedge t\neq y_n\wedge O(s,t))),\ n\geq 1$

\vspace{2mm}
\noindent
$\mathbf{\lambda_{2n}}:\ \forall y_1\ldots\forall y_n\exists s(\neg O(s,y_1)\wedge\ldots\wedge\neg O(s,y_n)),\ n\geq 2$

\vspace{2mm}
\noindent
$\mathbf{\lambda_3}:\ \forall x\neg O(x,x)$

\vspace{2mm}
\noindent
$\mathbf{\lambda_4}:\ \forall x\forall y(O(x,y)\rightarrow O(y,x))$

\vspace{2mm}
\noindent
$\mathbf{\lambda_5}:\ \forall x\exists y O(x,y)$

\vspace{2mm}
\noindent
$\mathbf{\lambda_6}:\ \forall x\forall y\forall z\forall t(O(x,z)\wedge O(y,z)\wedge O(x,t)\rightarrow O(y,t))$

\begin{remark}
	We denote by $\mathcal{F}^2_{\mathbb{R}}$ the structure for the language $\mathcal{L}$, having universe the set of all lines in Euclidean plane. Clearly $\mathcal{F}^2_{\mathbb{R}}\models SAPP$. 
\end{remark}

\begin{remark}
	We consider the binary predicate symbols $P$ (meaning that two lines are parallel), $C$ (meaning that two lines intersect) and the ternary predicate symbol $\rho$ (three lines intersect at one point). We denote by $\mathcal{F}^2_{\mathbb{R},PCO\rho}$ the structure for the language $\langle\ ;\ ;P,C,O,\rho,=\rangle$, having universe the set of all lines in Euclidean plane. Clearly $\mathcal{F}^2_{\mathbb{R},PCO\rho}\models SAPP$. 
\end{remark}

\begin{remark}
	We will denote by $A$, $B$, $C\ldots$ the universes of the structures $\mathcal{A}$, $\mathcal{B}$, $\mathcal{C}\ldots$ 
\end{remark}

\begin{remark}	
	Let $\varphi$ be a first-order formula with free variables $x_1,\ldots,x_n$, $\mathcal{A}$ be a structure for the language of $\varphi$ and $a_1,\ldots,a_n\in A$. By $\mathcal{A}\models\varphi[a_1,\ldots,a_n]$ we denote that $\varphi$ is true in $\mathcal{A}$ under valuation, assigning $a_1,\ldots,a_n$ to $x_1,\ldots,x_n$.
\end{remark}

\begin{remark}	
	We will denote $a_1,\ldots,a_n$ by $\overline{a}$ and $f(a_1),\ldots,f(a_n)$ by $\overline{f(a)}$. 
\end{remark}

\begin{proposition}
	Let $\mathcal{A}$ be a model of $SAPP$. We consider the relation $\mathbf{R_1}$, defined in the following way:\[xR_1 y\stackrel{def}{\Longleftrightarrow} \textmd{for every }z,\ \neg O(x,z)\textmd{ or }O(y,z)\] (Intuitively $xR_1 y$ means "$x$ does not intersect $y$".)
	
	$R_1$ is an equivalence relation which divides $A$ into infinitely many infinite equivalence classes.
\end{proposition}
\begin{proof}
	Obviously $R_1$ is reflexive.
	
	Let $x,y\in A$ and $xR_1 y$. We will prove $yR_1 x$. Let $z\in A$. We must prove that $\neg O(y,z)$ or $O(x,z)$. Let $O(y,z)$. We will show that $O(x,z)$. We have $\mathcal{A}\models\lambda_5$ and hence - there is $t\in A$ such that $O(x,t)$. From here and $xR_1 y$ we get that $O(y,t)$. From $O(y,t)$, $O(x,t)$, $O(y,z)$ and $\mathcal{A}\models\lambda_6$ we obtain that $O(x,z)$. Consequently $R_1$ is symmetric.
	
	It can be easily verified that $R_1$ is transitive.
	
	Suppose for the sake of contradiction that there are finitely many equivalence classes - $[x_1],\ldots, [x_n]$. Using that $\mathcal{A}\models\lambda_{2n}$ and $\mathcal{A}\models\lambda_5$, we get that there are $s,z\in A$ such that $\neg O(s,x_1),\ldots,\neg O(s,x_n)$, $O(s,z)$. Clearly $z\overline{R_1}x_1$, $z\overline{R_1}x_2,\ldots, z\overline{R_1}x_n$ and hence $[z]\neq [x_1],\ldots,[x_n]$ - a contradiction. Consequently there are infinitely many equivalence classes.
	
	Let $[x]$ be an arbitrary equivalence class. Suppose for the sake of contradiction that $[x]$ is finite, i.e. $[x]=\{y_1,\ldots,y_n\}$ for some $y_1,\ldots,y_n\in A$. From $\mathcal{A}\models\lambda_5$ we obtain that there is $y$ such that $O(y_1,y)$. Let $i\in \{2,\ldots,n\}$. From $y_1 R_1 y_i$ and $O(y_1,y)$ we conclude that $O(y_i,y)$. Using $\mathcal{A}\models\lambda_{1n}$, we get that there is $t\in A$ such that $t\neq y_1,\ldots,t\neq y_n$, $O(y,t)$. We will prove that $y_1 R_1 t$. Let $z\in A$. It suffices to prove that $O(y_1,z)$ implies $O(t,z)$. From $O(y_1,y)$, $O(t,y)$ and $\mathcal{A}\models\lambda_6$ we conclude that $O(t,z)$. Consequently $y_1 R_1 t$, i.e. $t\in [x]$, but we proved that $t\neq y_1,\ldots,y_n$ - a contradiction. Consequently $[x]$ is infinite. $\Box$
\end{proof}
 
\begin{definition}
	Let $\mathcal{A}$ be a model of $SAPP$, $R_1$ be the relation from the previous proposition. We consider the equivalence classes modulo $R_1$ and the following relation:\[[x]\mathbf{R_2} [y]\stackrel{def}{\Longleftrightarrow}O(x,y)\]
\end{definition}

It can be easily verified that this definition is correct and the following proposition holds:

\begin{proposition}\label{p2}
	Let $\mathcal{A}$ be a model of $SAPP$. Then for every equivalence class modulo $R_1$ $[x]$ there is exactly one equivalence class $[y]$ such that $[x]R_2 [y]$. Moreover $[x]\neq [y]$.
\end{proposition}

\begin{proposition}\label{p3}
	Let $\mathcal{A}$ be a countable model of $SAPP$, $\mathcal{B}$ be a model of $SAPP$. Then $\mathcal{A}$ is elementary embedded in $\mathcal{B}$.
\end{proposition}

\begin{proof}
	Since $\mathcal{A}$ is a countable model of $SAPP$, $A$ has countably many equivalence classes. Let these equivalence classes be $[a_1]$, $[a_2],\ldots,[a_n],\ldots,[a_1']$, $[a_2'],\ldots,[a_n'],\ldots$, all they being different and $[a_1]R_2 [a_1']$, $[a_2]R_2 [a_2'],\ldots, [a_n] R_2 [a_n'],\\\ldots$. There exist countably many equivalence classes of $B$ $[b_1]$, $[b_2],\ldots,[b_n],\ldots;\ [b_1']$, $[b_2'],\ldots,[b_n'],\ldots$ such that all they are different and $[b_1]R_2 [b_1']$, $[b_2] R_2 [b_2'],\ldots,[b_n]R_2$\\ $[b_n'],\ldots$. For every $n$, $[a_n]$ and $[a_n']$ are countable; $[b_n]$ and $[b_n']$ are infinite. Consequently for every $n$, there are injections $h_n:\ [a_n]\rightarrow [b_n]$ and $h_n':\ [a_n']\rightarrow [b_n']$. 
	
	We define the mapping $f$:
	\[f(a) =\left\{
	\begin{array}{lll}
	h_n(a) & \quad &\textmd{if }a\in [a_n] \textmd{ for some }n\\
	h_n'(a) & \quad &\textmd{if }a\in [a_n'] \textmd{ for some }n 
	\end{array} \right.\]
	
	We will prove that $f$ is an elementary embedding of $\mathcal{A}$ in $\mathcal{B}$. By induction on $\varphi$ we will prove that for every formula $\varphi$, if $\varphi$ has free variables among $x_1,\ldots, x_n$ and $c_1,\ldots,c_n\in A$, then $\mathcal{A}\models\varphi [\overline{c}]\Leftrightarrow \mathcal{B}\models\varphi [\overline{f(c)}]$.\\
	1) The base of induction is trivial.\\
	2) $\varphi$ is $\neg\varphi_1$\\
	The proof is obvious.\\
	3) $\varphi$ is $\varphi_1\wedge\varphi_2$\\
	The proof is obvious.\\
	4) $\varphi$ is $\exists x \varphi_1$\\
	Let the free variables of $\varphi$ be among $x_1,\ldots,x_n$ and $c_1,\ldots, c_n\in A$. It can be easily verified that $\mathcal{A}\models\exists x \varphi_1[\overline{c}]$ implies $\mathcal{B}\models\exists x \varphi_1[\overline{f(c)}]$.
	
	Let $\mathcal{B}\models\exists x\varphi_1[\overline{f(c)}]$. We will prove $\mathcal{A}\models\exists x\varphi_1[\overline{c}]$. There exists $b\in B$ such that $\mathcal{B}\models\varphi_1[b,f(c_1),\ldots,f(c_n)]$\\
	\textbf{Case 1}: $b=f(a)$ for some $a\in A$\\
	From the induction hypothesis and $\mathcal{B}\models\varphi_1[f(a),f(c_1),\ldots,f(c_n)]$ we get that $\mathcal{A}\models\varphi_1[a,c_1,\ldots,c_n]$ and hence $\mathcal{A}\models\exists x\varphi_1[c_1,\ldots,c_n]$.\\
	\textbf{Case 2}: $b\neq f(a)$ for every $a\in A$\\
	\textbf{Case 2.1}: $b\in [b_i]$ or $b\in [b_i']$ for some $i$\\
	Without loss of generality $b\in [b_i]$ for some $i$. There is $b'\in [b_i]$ such that $b'=f(a)$ for some $a\neq c_1,\ldots,c_n$. Since $b$, $b'\in [b_i]$, $\neg O(b,b')$ (we use the definition of $R_2$ and Proposition~\ref{p2}). We define a function $g:\ B\rightarrow B$ in the following way:
	\[g(x) =\left\{
	\begin{array}{lll}
	b' & \quad &\textmd{if }x=b\\
	b & \quad &\textmd{if }x=b'\\
	x & \quad &\textmd{otherwise} 
	\end{array} \right.\]
	
	\noindent
	Obviously $g$ is a bijection. We will prove that $g$ is an automorphism. Let $d_1$, $d_2\in B$. It suffices to prove that $O(d_1,d_2)\Leftrightarrow O(g(d_1),g(d_2))$. We will consider only the case when $d_1\notin \{b,b'\}$, $d_2\in \{b,b'\}$. Without loss of generality $d_2=b$. It suffices to prove $O(d_1, d_2)\Leftrightarrow O(d_1, b')$. We have $bR_1 b'$ and using the definition of $R_1$ and its symmetry, we get that $O(d_1,b)\Leftrightarrow O(d_1,b')$. Consequently $g$ is an automorphism. 
	
	We have $\mathcal{B}\models\varphi_1[b,f(c_1),\ldots,f(c_n)]$. Consequently $\mathcal{B}\models\varphi_1[g(b),g(f(c_1)),\ldots,\\g(f(c_n))]$, i.e. $\mathcal{B}\models\varphi_1[b',f(c_1),\ldots,f(c_n)]$. By the induction hypothesis, $\mathcal{A}\models\varphi_1[a,c_1,\ldots,c_n]$ and hence $\mathcal{A}\models\exists x\varphi_1[c_1,\ldots,c_n]$.\\
	\textbf{Case 2.2}: $b\notin [b_i]$ and $b\notin [b_i']$ for every $i$\\
	Let $k_1,\ldots,k_l\in A$ $(l\leq n)$ be such that for every $i=1,\ldots,l$, $[k_i]R_2 [c_j]$ for some $j\in \{1,\ldots,n\}$. There is $a\in A$ such that $a\notin [c_1],\ldots,[c_n],\ [k_1],\ldots, [k_l]$. We add to the language $\mathcal{L}$ the constants $d$, $d_1,\ldots,d_n$. Interpreting the new constants by $b$, $f(c_1),\ldots,f(c_n)$ or by $f(a)$, $f(c_1),\ldots,f(c_n)$, we obtain two structures for the extended language which we denote by $\mathcal{B}'=(\mathcal{B},b,f(c_1),\ldots,f(c_n))$ and $\mathcal{B}''=(\mathcal{B},f(a),f(c_1),\ldots,f(c_n))$ correspondingly. Let $[m_1]$ be the only equivalence class for which $[b]R_2 [m_1]$, and $[m_2]$ be the only equivalence class for which $[f(a)] R_2 [m_2]$. We consider Ehrenfeucht-Fra\"{\i}ss\`{e}'s game with arbitrary finite length $s$ and the following strategy for the second player: if the first player chooses $b$ $(f(a))$, then the second player chooses $f(a)$ $(b)$. Otherwise, if the first player chooses an element out of $[b]\cup [f(a)]\cup [m_1]\cup [m_2]$, then the second player chooses the same element; if the first player chooses a new element of $[b]$, then the second player chooses a new element of $[f(a)]$, different from $f(a)$ and the converse; if the first player chooses a new element of $[m_1]$, then the second player chooses a new element of $[m_2]$ and the converse; if the first player chooses already chosen in the corresponding structure element $x$ in $[b]\cup [f(a)]\cup [m_1]\cup [m_2]$, then the second player chooses the same element which then was chosen in the other structure. Let $e_1,\ldots,e_s$ and $e_1',\ldots,e_s'$ be correspondingly of $\mathcal{B}'$ and of $\mathcal{B}''$ in the order of their choosing. Let $h=\{\langle e_i,e_i'\rangle:\ i=1,\ldots,s\}\cup \{\langle C^{\mathcal{B}'}, C^{\mathcal{B}''}\rangle:\ C\textmd{ - a constant of the extended language}\}$. Let $\mathcal{B}_1'$ be the substructure of $\mathcal{B}'$, generated by $e_1,\ldots,e_s$, $b$, $f(c_1),\ldots,f(c_n)$; $\mathcal{B}_1''$ be the substructure of $\mathcal{B}''$, generated by $e_1',\ldots,e_s'$, $f(a)$, $f(c_1),\ldots,f(c_n)$. It can be easily verified that $h$ is an isomorphism from $\mathcal{B}_1'$ to $\mathcal{B}_1''$. Consequently for every closed formula $\varphi$ we have $\mathcal{B}'\models\varphi\Leftrightarrow \mathcal{B}''\models\varphi$ and hence $\mathcal{B}'\models\varphi_1(d,d_1,\ldots,d_n)\Leftrightarrow \mathcal{B}''\models\varphi_1(d,d_1,\ldots,d_n)$; so $\mathcal{B}\models\varphi_1[b,f(c_1),\ldots,f(c_n)]\Leftrightarrow \mathcal{B}\models\varphi_1[f(a),f(c_1),\ldots,f(c_n)]$. But we have $\mathcal{B}\models\varphi_1[b,f(c_1),\ldots,f(c_n)]$ and thus $\mathcal{B}\models\varphi_1[f(a),f(c_1),\ldots,f(c_n)]$. From the induction hypothesis, $\mathcal{A}\models\varphi_1[a,c_1,\ldots,c_n]$, i.e. $\mathcal{A}\models\exists x\varphi_1[\overline{c}]$. Consequently $f$ is an elementary embedding of $\mathcal{A}$ in $\mathcal{B}$. $\Box$
\end{proof}

\begin{corollary}\label{cor1}
	The theory $SAPP$ is complete.
\end{corollary}
\begin{proof}
	Let $\mathcal{F}^2_{\mathbb{Q}}$ be the structure for the language $\mathcal{L}$ with universe $\{\bm{a}\textmd{ - line in}\\\textmd{Euclidean plane: at least 2 of the points of }\bm{a}\textmd{ are with rational coordinates}\}$. The predicate symbol $O$ is interpreted by perpendicularity. We will prove that $\mathcal{F}^2_{\mathbb{Q}}\models\lambda_5$. Let $x_{1,2}$, $y_{1,2}\in\mathbb{Q}$ and $a$ be the line, determined by the points $(x_1,y_1)$, $(x_2,y_2)$.\\
	\textbf{Case 1}: $a$ does not coincide with and is not parallel to the ordinate axis\\
	Let $b$ be an arbitrary line in Euclidean plane, perpendicular to $a$, and $k$ be the slope of $b$. Let $k_1$ be the slope of $a$. The line $m$ with equation $y=kx$ is perpendicular to $a$. We have $k=-\frac{1}{k_1}=-\frac{x_2-x_1}{y_2-y_1}\in\mathbb{Q}$. At least two of the points of the line $m$ are with rational coordinates - $(0,0)$ and $(1,k)$.\\
	\textbf{Case 2}: $a$ coincides with or is parallel to the ordinate axis\\
	The abscissa axis is in $F^2_{\mathbb{Q}}$. Consequently $\mathcal{F}^2_{\mathbb{Q}}\models\lambda_5$. In a similar way it can be proved that $\mathcal{F}^2_{\mathbb{Q}}\models\lambda_{1n}$ for every $n\geq 1$ and $\mathcal{F}^2_{\mathbb{Q}}\models\lambda_{2n}$ for every $n\geq 2$. Clearly in $\mathcal{F}^2_{\mathbb{Q}}$ are true also $\lambda_3$, $\lambda_4$ and $\lambda_6$. Consequently $\mathcal{F}^2_{\mathbb{Q}}\models SAPP$. 
	
	Let $\mathcal{A}$ and $\mathcal{B}$ be arbitrary models of $SAPP$. Since $\mathcal{F}^2_{\mathbb{Q}}$ is a countable model of $SAPP$, $\mathcal{F}^2_{\mathbb{Q}}$ is elementary embedded in $\mathcal{A}$ and in $\mathcal{B}$ (Proposition~\ref{p3}). Consequently $\mathcal{F}^2_{\mathbb{Q}}$ is elementary equivalent to $\mathcal{A}$ and to $\mathcal{B}$ and hence $\mathcal{A}$ and $\mathcal{B}$ are elementary equivalent. Consequently $SAPP$ is complete. $\Box$
\end{proof}

\begin{proposition}
	\textbf{(i)} $=$ is not definable by $O$ in every model of $SAPP$;\\
	\textbf{(ii)} $O$ is not definable by $=$ in every model of $SAPP$;\\
	\textbf{(iii)} the ternary predicate $\rho(a,b,c)$ is not definable in $\mathcal{F}^2_{\mathbb{R},PCO\rho}$ by $O$ and $=$;\\
	\textbf{(iv)} the binary predicate $P(a,b)$ is definable in $\mathcal{F}^2_{\mathbb{R},PCO\rho}$ by $O$ and $=$;\\
	\textbf{(v)} the binary predicate $C(a,b)$ is definable in $\mathcal{F}^2_{\mathbb{R},PCO\rho}$ by $O$ and $=$;\\
	\textbf{(vi)} $O$ is not definable in $\mathcal{F}^2_{\mathbb{R},PCO\rho}$ by $P$, $C$ and $=$.
\end{proposition}
\begin{proof}
	
	\textbf{(i)} Let $\mathcal{A}\models SAPP$. Suppose for the sake of contradiction that there is a formula $\varphi(x,y)$ in which occurs only the predicate symbol $O$, such that $\mathcal{A}\models\forall x\forall y(x=y\leftrightarrow \varphi(x,y))$. Let $a\in A$ and $g:\ [a]\rightarrow [a]$ be a surjection  which is not an injection. We define $f:\ A\rightarrow A$ in the following way:
	\[f(b) =\left\{
	\begin{array}{lll}
	b & \quad &\textmd{if }b\notin [a]\\
	g(b) & \quad &\textmd{otherwise }
	\end{array} \right.\]
	
	By induction on $\psi$ it can be easily proved that for every formula $\psi$, if in $\psi$ do not occur other predicate symbols except $O$, $\psi$ has free variables among $x_1,\ldots, x_n$ and $a_1,\ldots,a_n\in A$, then $\mathcal{A}\models\psi[\overline{a}]$ iff $\mathcal{A}\models\psi[\overline{f(a)}]$. Consequently for all $x'$, $y'\in A$, $\mathcal{A}\models\varphi[x',y']$ iff $\mathcal{A}\models\varphi[f(x'),f(y')]$.
	
	Since $g$ is not an injection, there are $b_1$, $b_2\in [a]$ such that $b_1\neq b_2$ and $g(b_1)=g(b_2)$. Consequently $f(b_1)=f(b_2)$. Clearly we have the following equivalences: $b_1=b_2$ iff $\mathcal{A}\models\varphi[b_1,b_2]$ iff $\mathcal{A}\models\varphi[f(b_1),f(b_2)]$ iff $f(b_1)=f(b_2)$. Consequently $b_1=b_2$ - a contradiction.
	
	\textbf{(ii)} Let $\mathcal{A}\models SAPP$. Suppose for the sake of contradiction that there is a formula $\varphi(x,y)$ in which occurs only the predicate symbol $=$, such that $\mathcal{A}\models\forall x\forall y(O(x,y)\leftrightarrow \varphi(x,y))$.
	
	We consider $\mathcal{F}^2_{\mathbb{Q}}$. Let $a$ be a line in the plane of which at least two of the points have rational coordinates, and $[b]$ be the only equivalence class such that $[a]R_2[b]$. Let $[c]\neq [a]$ and $[c]\neq [b]$. Let $[a]=\{a_0,a_1,\ldots,a_n,\ldots\}$ and $[c]=\{c_0,c_1,\ldots,c_n,\ldots\}$. We define a function $f:\ F^2_{\mathbb{Q}}\rightarrow F^2_{\mathbb{Q}}$ in the following way:
	\[f(x) =\left\{
	\begin{array}{lll}
	x & \quad &\textmd{if }x\notin[a]\textmd{ and }x\notin[c]\\
	c_i & \quad &\textmd{if }x=a_i\textmd{ for some }i\\
	a_i & \quad &\textmd{if }x=c_i\textmd{ for some }i
	\end{array} \right.\]
	
	By induction on $\psi$ it can be easily verified that for every formula $\psi$, if in $\psi$ do not occur other predicate symbols except $=$, $\psi$ has free variables among $x_1,\ldots,x_n$ and $a_1,\ldots,a_n\in F^2_{\mathbb{Q}}$, then $\mathcal{F}^2_{\mathbb{Q}}\models\psi[\overline{a}]$ iff $\mathcal{F}^2_{\mathbb{Q}}\models\psi[\overline{f(a)}]$.
	
	$\mathcal{F}^2_{\mathbb{Q}}$ is countable and $\mathcal{F}^2_{\mathbb{Q}}\models SAPP$ (it was proved in Corollary~\ref{cor1}); $\mathcal{A}\models SAPP$; so by Proposition~\ref{p3} we obtain that $\mathcal{F}^2_{\mathbb{Q}}$ is elementary embedded in $\mathcal{A}$. Consequently $\mathcal{F}^2_{\mathbb{Q}}\models\forall x\forall y(O(x,y)\leftrightarrow \varphi(x,y))$. We have $[a]R_2[b]$, so $O(a,b)$. We have also $f(a)\in[c]$ and $[c]\neq[a]$, so $\neg O(f(a),b)$, i.e. $\neg O(f(a),f(b))$. Clearly we have the following equivalences: $O(a,b)$ iff $\mathcal{F}^2_{\mathbb{Q}}\models\varphi[a,b]$ iff $\mathcal{F}^2_{\mathbb{Q}}\models\varphi[f(a),f(b)]$ iff $O(f(a),f(b))$. Consequently $O(f(a),f(b))$ - a contradiction.
	
	\textbf{(iii)} Suppose for the sake of contradiction that there is a formula $\varphi(x,y,z)$ in $\mathcal{L}$ such that $\mathcal{F}^2_{\mathbb{R},PCO\rho}\models\forall x\forall y\forall z(\rho(x,y,z)\leftrightarrow\varphi(x,y,z))$. Let $a$, $b$ and $c$ be three lines in the plane which intersect at one point. Let $a'$ be parallel to $a$. Clearly $\mathcal{F}^2_{\mathbb{R},PCO\rho}\models\varphi[a,b,c]$ and $\mathcal{F}^2_{\mathbb{R},PCO\rho}\not\models\varphi[a',b,c]$. We define $h:\ F^2_{\mathbb{R}}\rightarrow F^2_{\mathbb{R}}$ in the following way:
	\[h(x) =\left\{
	\begin{array}{lll}
	x & \quad &\textmd{if }x\neq a,a'\\
	a' & \quad &\textmd{if }x=a\\
	a & \quad &\textmd{if }x=a'
	\end{array} \right.\]
	It can be easily verified that $h$ is a bijection and preserves the predicates $O$ and $=$. Consequently $\mathcal{F}^2_{\mathbb{R},PCO\rho}\models\varphi[a,b,c]$ iff $\mathcal{F}^2_{\mathbb{R},PCO\rho}\models\varphi[h(a),h(b),h(c)]$. Thus $\mathcal{F}^2_{\mathbb{R},PCO\rho}\models\varphi[h(a),h(b),h(c)]$, i.e. $\mathcal{F}^2_{\mathbb{R},PCO\rho}\models\varphi[a',b,c]$ - a contradiction.
	
	\textbf{(iv)} We have $\mathcal{F}^2_{\mathbb{R},PCO\rho}\models\forall x\forall y(P(x,y)\leftrightarrow(x\neq y\wedge\exists z(O(x,z)\wedge O(y,z))))$. 
	
	\textbf{(v)} We have $\mathcal{F}^2_{\mathbb{R},PCO\rho}\models\forall x\forall y(C(x,y)\leftrightarrow\exists z(O(x,z)\wedge\neg O(y,z)))$. 
	
	\textbf{(vi)} The proof is similar as for (ii). $\Box$
\end{proof}

\section{Decidability and complexity of the first-order theory}

\begin{proposition}
	The problem if a closed formula in $\mathcal{L}$ logically follows from $SAPP$ is PSPACE-complete.
\end{proposition}
\begin{proof}
	We consider $EQ^{\infty}=\{\varphi:\ \varphi\textmd{ is a closed formula in the language }\mathcal{L}_1=\langle\ ;\ ;=\rangle\textmd{ and }\varphi\textmd{ is true in all infinite structures}\}$ \cite{BalbianiTinchev}. The membership problem in $EQ^{\infty}$ is PSPACE-complete \cite{BalbianiTinchev}.
	
	First we will prove that the problem if a closed formula in $\mathcal{L}$ logically follows from $SAPP$ is in PSPACE. It is enough to juxtapose to every closed formula $\varphi$ in $\mathcal{L}$ a closed formula $\varphi_1$ in $\mathcal{L}_1$ and to ensure that $\varphi_1$ can be obtained from $\varphi$ algorithmically with use of memory polynomial in the size of $\varphi$, and $SAPP\models\varphi$ iff $\varphi_1\in EQ^{\infty}$.
	
	Let $\bm{\mathcal{A}^{\ast}}$ be the substructure of $\mathcal{F}^2_{\mathbb{R}}$, which is obtained by eliminating of the lines parallel to or coinciding with the abscissa axis, the lines parallel to or coinciding with the ordinate axis and the lines with equation of the kind $y=bx$.
	
	\begin{lemma}
		The structure $\mathcal{A}^{\ast}$ is a model of $SAPP$.
	\end{lemma}
	\begin{proof}
		Let $n\geq 1$. We will prove $\mathcal{A}^{\ast}\models\lambda_{1n}$. Let $b_1,\ldots,b_n$, $c\in A^{\ast}$ and $O(c,b_1),\ldots, O(c,b_n)$. We will prove that there is $d\in A^{\ast}$ such that $d\neq b_1,\ldots,b_n$ and $O(c,d)$. We have that $b_1,\ldots,b_n$ are parallel and consequently have the same slope $m$. Let the $y$-intercepts of $b_1,\ldots,b_n$ be $q_1,\ldots,q_n$. Let $q\neq 0$, $q_1,\ldots,q_n$. Let $d$ be the line with equation $y=mx+q$. Clearly $d$ has the desired properties. 
		
		Let $n\geq 2$. We will prove that $\mathcal{A}^{\ast}\models\lambda_{2n}$. Let $b_1,\ldots,b_n\in A^{\ast}$. We will prove that there is $c\in A^{\ast}$ such that $\neg O(c,b_1),\ldots,\neg O(c,b_n)$. Let the slopes of $b_1,\ldots,b_n$ be correspondingly $m_1,\ldots,m_n$. Let $m\neq 0$, $-\frac{1}{m_1},\ldots,-\frac{1}{m_n}$. Let $c$ be the line with equation $y=mx+1$. Clearly $c$ has the desired properties.
		
		Clearly the other formulas of $SAPP$ are true in $\mathcal{A}^{\ast}$. $\Box$
	\end{proof}
	
	\begin{lemma}\label{lemma3}
		For every closed formula $\varphi$ in the language $\mathcal{L}$, $SAPP\models\varphi$ iff $\mathcal{A}^{\ast}\models\varphi$.
	\end{lemma}
	\begin{proof}
		Clearly $SAPP\models\varphi$ implies $\mathcal{A}^{\ast}\models\varphi$. Let $\mathcal{A}^{\ast}\models\varphi$ and $\mathcal{B}\models SAPP$. We will prove that $\mathcal{B}\models\varphi$. $\mathcal{A}^{\ast}$ is a model of $SAPP$ and $SAPP$ is complete (Corollary~\ref{cor1}); so $\mathcal{A}^{\ast}$ is elementary equivalent to every model of $SAPP$; so $\mathcal{A}^{\ast}$ is elementary equivalent to $\mathcal{B}$. Consequently $\mathcal{B}\models\varphi$. $\Box$
	\end{proof}
	
	Let $\bm{\mathcal{R}}$ be the structure for $\mathcal{L}_1$ with universe $\mathbb{R}\setminus\{0\}$.
	
	\begin{lemma}\label{lemma4}
		For any closed formula $\varphi$ in $\mathcal{L}_1$, $\mathcal{R}\models\varphi$ iff $\varphi\in EQ^{\infty}$.
	\end{lemma}
	\begin{proof}
		Clearly $\varphi\in EQ^{\infty}$ implies $\mathcal{R}\models\varphi$. Let $\mathcal{R}\models\varphi$. We will prove $\varphi\in EQ^{\infty}$. Let $\mathcal{A}$ be an infinite structure. We will prove that $\mathcal{A}\models\varphi$. We consider the theory $T=\{\varphi\}$. Let $C$ be a set of constants such that the cardinality of $C$ coincides with the cardinality of $A$. $T$ has an infinite model ($\mathcal{R}$) and can be considered as a theory in the language $\mathcal{L}_2=\langle C;\ ;=\rangle$ with cardinality - the cardinality of $A$. Thus by L\"{o}wenheim-Skolem theorem we obtain that $T$ has a model $\mathcal{B}$ for the language $\mathcal{L}_2$ such that the cardinality of $B$ is the maximal of $\omega$ and the cardinality of $A$, i.e. the cardinality of $A$. We consider $\mathcal{B}$ as a structure for $\mathcal{L}_1$ and since in $\varphi$ do not occur constants, we have again $\mathcal{B}\models\varphi$. $\mathcal{A}$ and $\mathcal{B}$ are structures for $\mathcal{L}_1$ and the cardinality of $A$ coincides with the cardinality of $B$; so $\mathcal{A}$ and $\mathcal{B}$ are isomorphic; so $\mathcal{A}\models\varphi$. $\Box$
	\end{proof}
	
	Let $a$ be a line with equation $y=bx+c$. We use the following notations: $\bm{a^1}=b$, $\bm{a^2}=-\frac{1}{b}$, $\bm{a^3}=c$. It is convenient to call $a^1$, $a^2$ and $a^3$ \textit{coordinates} of the line $a$.
	
	To every formula $\varphi$ in the language $\mathcal{L}$ we juxtapose a formula $\widehat{\varphi}$ in the language $\mathcal{L}_1$ in the following way:\\
	\textbf{1) }$\bm{\varphi}$\textbf{ - atomic}\\
	$(\bullet)$ If $\varphi$ is $O(x_1,x_2)$, then $\widehat{\varphi}$ is $x^1_1=x^2_2$.\\
	$(\bullet)$ If $\varphi$ is $x_1=x_2$, then $\widehat{\varphi}$ is $(x^1_1=x^1_2)\wedge(x^3_1=x^3_2)$.
	
	\vspace{2mm}
	\noindent
	\textbf{2) }$\bm{\varphi}$\textbf{ is }$\bm{\neg\varphi'}$\\
	$\widehat{\varphi}$ is $\neg\widehat{\varphi'}$.

	\vspace{2mm}
	\noindent
	\textbf{3) }$\bm{\varphi}$\textbf{ is }$\bm{\varphi'\wedge\varphi''}$\\
	$\widehat{\varphi}$ is $\widehat{\varphi'}\wedge\widehat{\varphi''}$.

	\vspace{2mm}
	\noindent
	\textbf{4) }$\bm{\varphi}$\textbf{ is }$\bm{\exists x_n\varphi'}$\textbf{ and }$\bm{\varphi'}$\textbf{ has free variables }$\bm{x_1,\ldots,x_n}$\\
	$\widehat{\varphi}$ is $\exists x^1_n\exists x^2_n\exists x^3_n(\widehat{\varphi'}\wedge\kappa_n)$, where for every natural number $n$, $\kappa_n$ is a formula with free variables $x^1_1$, $x^2_1,\ldots,x^1_n$, $x^2_n$, defined in the following way:
	\[\kappa_n:\ x^1_n\neq x^2_n\wedge\bigwedge_{i<n}(x^1_i=x^1_n\leftrightarrow x^2_i=x^2_n)\wedge\bigwedge_{i<n}(x^1_i=x^2_n\leftrightarrow x^2_i=x^1_n)\]

	\begin{definition}\label{def2}
		Let $n$ be a natural number, $a^1_1$, $a^2_1$, $a^3_1,\ldots,a^1_n$, $a^2_n$, $a^3_n\in\mathbb{R}\setminus\{0\}$ and for every $i\in\{1,\ldots,n\}$, $\mathcal{R}\models\kappa_i[a^1_1,a^2_1,\ldots,a^1_i,a^2_i]$. We say that $b^1_1$, $b^2_1$, $b^3_1,\ldots,b^1_n$, $b^2_n$, $b^3_n$ are \textit{corresponding} to $a^1_1$, $a^2_1$, $a^3_1,\ldots,a^1_n$, $a^2_n$, $a^3_n$ if for every $i\in\{1,\ldots,n\}$,
	
		\vspace{2mm}
		\noindent
		\textbf{1)} $b^1_i$, $b^2_i$, $b^3_i\in\mathbb{R}\setminus\{0\}$
	
		\vspace{2mm}
		\noindent
		\textbf{2)} if $a^1_i\notin\{a^1_1,a^2_1,\ldots,a^1_{i-1}, a^2_{i-1}\}$, then $b^1_i\notin\{b^1_1,b^2_1,\ldots,b^1_{i-1},b^2_{i-1}\}$
	
		\vspace{2mm}
		\noindent
		\textbf{3)} if $a^1_i=a^1_k$ for some $k\in\{1,\ldots,i-1\}$, then $b^1_i=b^1_k$
	
		\vspace{2mm}
		\noindent
		\textbf{4)} if $a^1_i=a^2_k$ for some $k\in\{1,\ldots,i-1\}$, then $b^1_i=b^2_k$
	
		\vspace{2mm}
		\noindent
		\textbf{5)} $b^2_i=-\frac{1}{b^1_i}$
	
		\vspace{2mm}
		\noindent
		\textbf{6)} $b^3_i=a^3_i$
	
	\end{definition}

	By induction on $j$ it can be easily proved the following 

	\begin{lemma}\label{lemma0}
		Let $b^1_1$, $b^2_1$, $b^3_1,\ldots,b^1_n$, $b^2_n$, $b^3_n$ be corresponding to $a^1_1$, $a^2_1$, $a^3_1,\ldots,a^1_n$, $a^2_n$, $a^3_n$. Then for any $j$ and $i$, 
		if $1\leq i<j\leq n$, then\\
		1) $b^1_j=b^1_i$ implies $a^1_j=a^1_i$;\\
		2) $b^1_j=b^2_i$ implies $a^1_j=a^2_i$.
	\end{lemma}

	\begin{lemma}\label{lemma5}
		Let $b^1_1$, $b^2_1$, $b^3_1,\ldots,b^1_n$, $b^2_n$, $b^3_n$ be corresponding to $a^1_1$, $a^2_1$, $a^3_1,\ldots,a^1_n$, $a^2_n$, $a^3_n$ and $\varphi$ be $O(x,y)$ or $x=y$. Then for any $i$, $j\in\{1,\ldots,n\}$, $\mathcal{R}\models\widehat{\varphi}[a^1_i,a^2_i,a^3_i,a^1_j,a^2_j,a^3_j]$ iff $\mathcal{R}\models\widehat{\varphi}[b^1_i,b^2_i,b^3_i,b^1_j,b^2_j,b^3_j]$.
	\end{lemma}
	\begin{proof}
		Let $i$, $j\in\{1,\ldots,n\}$.\\
		\textbf{Case 1}: $i=j$\\
		The proof is trivial.\\
		\textbf{Case 2}: $i\neq j$\\
		Without loss of generality $i<j$.\\
		\textbf{Case 2.1}: $\varphi$ is $O(x,y)$\\
		$\widehat{\varphi}$ is $x^1=y^2$. We must prove the following two equivalences:\\
		1) $a^1_i=a^2_j$ iff $b^1_i=b^2_j$\\
		2) $a^1_j=a^2_i$ iff $b^1_j=b^2_i$\\
	
		Since $\mathcal{R}\models\kappa_j[a^1_1,a^2_1,\ldots,a^1_j,a^2_j]$, $a^1_i=a^2_j$ iff $a^2_i=a^1_j$. Since $b^1_j\cdot b^2_j=-1$ and $b^1_i\cdot b^2_i=-1$, $b^1_i=b^2_j$ iff $b^1_j=b^2_i$. Consequently it suffices to prove the second equivalence. By condition 4) of Definition~\ref{def2}, $a^1_j=a^2_i$ implies $b^1_j=b^2_i$. By Lemma~\ref{lemma0}, $b^1_j=b^2_i$ implies $a^1_j=a^2_i$.\\
		\textbf{Case 2.2}: $\varphi$ is $x=y$\\
		$\widehat{\varphi}$ is $(x^1=y^1)\wedge (x^3=y^3)$.\\
		We must prove the following equivalence:\\
		$(a^1_i=a^1_j\textmd{ and }a^3_i=a^3_j)$ iff $(b^1_i=b^1_j\textmd{ and }b^3_i=b^3_j)$\\
		By conditions 3) and 6) of Definition~\ref{def2} we get that $(a^1_i=a^1_j\textmd{ and }a^3_i=a^3_j)$ implies $(b^1_j=b^1_i\textmd{ and }b^3_i=b^3_j)$.\\
		By Lemma~\ref{lemma0} and condition 6) of Definition~\ref{def2} we get that $(b^1_j=b^1_i\textmd{ and }b^3_i=b^3_j)$ implies $(a^1_j=a^1_i\textmd{ and }a^3_i=a^3_j)$. $\Box$
	\end{proof}
	
	\begin{lemma}\label{lemma6}
		Let $b^1_1$, $b^2_1$, $b^3_1,\ldots,b^1_n$, $b^2_n$, $b^3_n$ be corresponding to $a^1_1$, $a^2_1$, $a^3_1,\ldots,a^1_n$, $a^2_n$, $a^3_n$ and $1\leq j_1\leq j_2\leq\ldots\leq j_m\leq n$. Then $b^1_{j_1}$, $b^2_{j_1}$, $b^3_{j_1},\ldots,b^1_{j_m}$, $b^2_{j_m}$, $b^3_{j_m}$ are corresponding to $a^1_{j_1}$, $a^2_{j_1}$, $a^3_{j_1},\ldots,a^1_{j_m}$, $a^2_{j_m}$, $a^3_{j_m}$.
	\end{lemma}
	\begin{proof}
		Since $b^1_1$, $b^2_1$, $b^3_1,\ldots,b^1_n$, $b^2_n$, $b^3_n$ are corresponding to $a^1_1$, $a^2_1$, $a^3_1,\ldots,a^1_n$, $a^2_n$, $a^3_n$, for every $i\in\{1,\ldots,m\}$,\\
		1) $\mathcal{R}\models\kappa_i[a^1_{j_1},a^2_{j_1},\ldots,a^1_{j_i},a^2_{j_i}]$;
		
		\vspace{2mm}
		\noindent
		2) $b^1_{j_i}$, $b^2_{j_i}$, $b^3_{j_i}\in\mathbb{R}\setminus\{0\}$;
		
		\vspace{2mm}
		\noindent
		3) $b^2_{j_i}=-\frac{1}{b^1_{j_i}}$;
		
		\vspace{2mm}
		\noindent
		4) $b^3_{j_i}=a^3_{j_i}$.
		
		\vspace{2mm}
		\noindent
		Let $i\in\{1,\ldots,m\}$. We consider\\
		\textbf{Case 1}: $a^1_{j_i}\notin\{a^1_{j_1},a^2_{j_1},\ldots,a^1_{j_{i-1}},a^2_{j_{i-1}}\}$\\
		We will prove that $b^1_{j_i}\notin\{b^1_{j_1},b^2_{j_1},\ldots,b^1_{j_{i-1}},b^2_{j_{i-1}}\}$. Suppose for the sake of contradiction the converse. By Lemma~\ref{lemma0}, $a^1_{j_i}\in\{a^1_{j_1},a^2_{j_1},\ldots,a^1_{j_{i-1}},a^2_{j_{i-1}}\}$ which is the desired contradiction.\\
		\textbf{Case 2}: $a^1_{j_i}\in\{a^1_{j_1},a^2_{j_1},\ldots,a^1_{j_{i-1}},a^2_{j_{i-1}}\}$\\
		Clearly:\\
		1) if $a^1_{j_i}=a^1_{j_k}$ for some $k\in\{1,\ldots,i-1\}$, then $b^1_{j_i}=b^1_{j_k}$;\\
		2) if $a^1_{j_i}=a^2_{j_k}$ for some $k\in\{1,\ldots,i-1\}$, then $b^1_{j_i}=b^2_{j_k}$.\\
		$\Box$
	\end{proof}

	\begin{lemma}\label{lemma7}
		Let $\varphi$ be a formula for $\mathcal{L}$ with free variables $x_1,\ldots,x_m$. Let $b^1_1$, $b^2_1$, $b^3_1,\ldots,b^1_n$, $b^2_n$, $b^3_n$ be corresponding to $a^1_1$, $a^2_1$, $a^3_1,\ldots,a^1_n$, $a^2_n$, $a^3_n$. Then for any $j_1,\ldots,j_m\in\{1,\ldots,n\}$, $\mathcal{R}\models\widehat{\varphi}[a^1_{j_1}, a^2_{j_1}, a^3_{j_1},\ldots,a^1_{j_m}, a^2_{j_m}, a^3_{j_m}]$ iff $\mathcal{R}\models\widehat{\varphi}[b^1_{j_1}, b^2_{j_1}, b^3_{j_1},\ldots,b^1_{j_m}, b^2_{j_m}, b^3_{j_m}]$.
	\end{lemma}
	\begin{proof}
		Induction on $\varphi$. For the base of induction we use Lemma~\ref{lemma5}. If $\varphi$ is $\neg\varphi'$ or $\varphi'\wedge\varphi''$, the proof is trivial. Let $\varphi$ be $\exists x_{m+1}\varphi'$ and $\varphi'$ has free variables $x_1,\ldots,x_{m+1}$. Let $b^1_1$, $b^2_1$, $b^3_1,\ldots,b^1_n$, $b^2_n$, $b^3_n$ be corresponding to $a^1_1$, $a^2_1$, $a^3_1,\ldots,a^1_n$, $a^2_n$, $a^3_n$. Let $j_1,\ldots,j_m\in\{1,\ldots,n\}$. Without loss of generality $j_1\leq\ldots\leq j_m$. We will prove that $\mathcal{R}\models\exists x^1_{m+1}\exists x^2_{m+1}\exists x^3_{m+1}(\widehat{\varphi'}\wedge\kappa_{m+1})[a^1_{j_1},a^2_{j_1},a^3_{j_1},\ldots,a^1_{j_m},a^2_{j_m},a^3_{j_m}]$ iff\\
		$\mathcal{R}\models\exists x^1_{m+1}\exists x^2_{m+1}\exists x^3_{m+1}(\widehat{\varphi'}\wedge\kappa_{m+1})[b^1_{j_1},b^2_{j_1},b^3_{j_1},\ldots,b^1_{j_m},b^2_{j_m},b^3_{j_m}]$.
		
		\vspace{2 mm}
		\noindent
		$\bm{\Rightarrow)}$ There are $c^1$, $c^2$, $c^3\in\mathbb{R}\setminus\{0\}$ such that \\
		$\mathcal{R}\models(\widehat{\varphi'}\wedge\kappa_{m+1})[a^1_{j_1},a^2_{j_1},a^3_{j_1},\ldots,a^1_{j_m},a^2_{j_m},a^3_{j_m},c^1,c^2,c^3]$.
		
		\begin{itemize}
		\item If $c^1\notin\{a^1_{j_1},a^2_{j_1},\ldots,a^1_{j_m},a^2_{j_m}\}$, we choose a real number $d^1$, different from $0$, such that $d^1\notin\{b^1_{j_1},b^2_{j_1},\ldots,b^1_{j_m},b^2_{j_m}\}$.
		\item If $c^1=a^1_{j_k}$ for some $k\in\{1,\ldots,m\}$, we denote $d^1=b^1_{j_k}$.
		\item If $c^1=a^2_{j_k}$ for some $k\in\{1,\ldots,m\}$, we denote $d^1=b^2_{j_k}$.
		\end{itemize}
	
	\noindent
	We denote $d^2=-\frac{1}{d^1}$, $d^3=c^3$. Using Lemma~\ref{lemma6}, we get that $b^1_{j_1}$, $b^2_{j_1}$, $b^3_{j_1},\ldots,b^1_{j_m}$, $b^2_{j_m}$, $b^3_{j_m}$, $d^1$, $d^2$, $d^3$ are corresponding to $a^1_{j_1}$, $a^2_{j_1}$, $a^3_{j_1},\ldots,a^1_{j_m}$, $a^2_{j_m}$, $a^3_{j_m}$, $c^1$, $c^2$, $c^3$. By the induction hypothesis, $\mathcal{R}\models\widehat{\varphi'}[b^1_{j_1}, b^2_{j_1}, b^3_{j_1},\ldots,b^1_{j_m},b^2_{j_m},b^3_{j_m},\\
	d^1,d^2,d^3]$. Since for every $k\in\{1,\ldots,m\}$, $b^1_{j_k}\cdot b^2_{j_k}=-1$ and $d^1\cdot d^2=-1$,\\ $\mathcal{R}\models\kappa_{m+1}[b^1_{j_1}, b^2_{j_1}, b^3_{j_1},\ldots,b^1_{j_m},b^2_{j_m},b^3_{j_m},d^1,d^2,d^3]$.
	
	\noindent
	$\bm{\Leftarrow)}$ There are $d^1$, $d^2$, $d^3\in\mathbb{R}\setminus\{0\}$ such that $\mathcal{R}\models(\widehat{\varphi'}\wedge\kappa_{m+1})[b^1_{j_1}, b^2_{j_1}, b^3_{j_1},\ldots,b^1_{j_m},\\b^2_{j_m},b^3_{j_m},
	d^1,d^2,d^3]$. By Lemma~\ref{lemma6}, $b^1_{j_1}$, $b^2_{j_1}$, $b^3_{j_1},\ldots,b^1_{j_m}$, $b^2_{j_m}$, $b^3_{j_m}$ are corresponding to $a^1_{j_1}$, $a^2_{j_1}$, $a^3_{j_1},\ldots,a^1_{j_m}$, $a^2_{j_m}$, $a^3_{j_m}$ and hence $\mathcal{R}\models\kappa_k[a^1_{j_1},a^2_{j_1},\ldots,a^1_{j_k},a^2_{j_k}]$ for every $k\in\{1,\ldots,m\}$.\\
	\textbf{Case 1}: $d^1\notin\{b^1_{j_1},b^2_{j_1},\ldots,b^1_{j_m},b^2_{j_m}\}$\\
	Let $e^1$ be a real number, different from zero, such that $e^1\notin\{b^1_{j_1},b^2_{j_1},\ldots,b^1_{j_m},b^2_{j_m}\}$. We denote $e^2=-\frac{1}{e^1}$, $e^3=d^3$. It can be easily proved that $b^1_{j_1}$, $b^2_{j_1}$, $b^3_{j_1},\ldots,b^1_{j_m}$, $b^2_{j_m}$, $b^3_{j_m}$, $e^1$, $e^2$, $e^3$ are corresponding to $b^1_{j_1}$, $b^2_{j_1}$, $b^3_{j_1},\ldots,b^1_{j_m}$, $b^2_{j_m}$, $b^3_{j_m}$, $d^1$, $d^2$, $d^3$. By the induction hypothesis, $\mathcal{R}\models\widehat{\varphi'}[b^1_{j_1}, b^2_{j_1}, b^3_{j_1},\ldots,b^1_{j_m}, b^2_{j_m}, b^3_{j_m}, e^1, e^2, e^3]$.
	
	Let $c^1$ and $c^2$ be different real numbers such that $c^1$, $c^2\notin\{0,a^1_{j_1},a^2_{j_1},\ldots,a^1_{j_m},\\a^2_{j_m}\}$. We will prove that $b^1_{j_1}$, $b^2_{j_1}$, $b^3_{j_1},\ldots,b^1_{j_m}$, $b^2_{j_m}$, $b^3_{j_m}$, $e^1$, $e^2$, $e^3$ are corresponding to $a^1_{j_1}$, $a^2_{j_1}$, $a^3_{j_1},\ldots,a^1_{j_m}$, $a^2_{j_m}$, $a^3_{j_m}$, $c^1$, $c^2$, $e^3$. Obviously $\mathcal{R}\models\kappa_{m+1}[a^1_{j_1},a^2_{j_1},\ldots,a^1_{j_m},a^2_{j_m},c^1,c^2]$. Since $b^1_{j_1}$, $b^2_{j_1}$, $b^3_{j_1},\ldots,b^1_{j_m}$, $b^2_{j_m}$, $b^3_{j_m}$ are corresponding to $a^1_{j_1}$, $a^2_{j_1}$, $a^3_{j_1},\ldots,a^1_{j_m}$, $a^2_{j_m}$, $a^3_{j_m}$, we have that conditions $2)$, $3)$ and $4)$ of Definition~\ref{def2} are fulfilled for every step from $1$ to $m$. Obviously they are fulfilled also for step $m+1$. Clearly conditions $5)$ and $6)$ of Definition~\ref{def2} are fulfilled at every step. Thus $b^1_{j_1}$, $b^2_{j_1}$, $b^3_{j_1},\ldots,b^1_{j_m}$, $b^2_{j_m}$, $b^3_{j_m}$, $e^1$, $e^2$, $e^3$ are corresponding to $a^1_{j_1}$, $a^2_{j_1}$, $a^3_{j_1},\ldots,a^1_{j_m}$, $a^2_{j_m}$, $a^3_{j_m}$, $c^1$, $c^2$, $e^3$. By the induction hypothesis, $\mathcal{R}\models\widehat{\varphi'}[a^1_{j_1}, a^2_{j_1}, a^3_{j_1},\ldots,a^1_{j_m}, a^2_{j_m}, a^3_{j_m}, c^1, c^2, e^3]$ iff $\mathcal{R}\models\widehat{\varphi'}[b^1_{j_1}, b^2_{j_1}, b^3_{j_1},\ldots,b^1_{j_m}, b^2_{j_m}, b^3_{j_m}, e^1, e^2, e^3]$. Consequently $\mathcal{R}\models\widehat{\varphi'}[a^1_{j_1}, a^2_{j_1}, a^3_{j_1},\ldots,\\a^1_{j_m}, a^2_{j_m}, a^3_{j_m}, c^1, c^2, e^3]$. We have also that $\mathcal{R}\models\kappa_{m+1}[a^1_{j_1}, a^2_{j_1}, \ldots,a^1_{j_m}, a^2_{j_m}, c^1, c^2]$. Thus $\mathcal{R}\models(\widehat{\varphi'}\wedge\kappa_{m+1})[a^1_{j_1}, a^2_{j_1}, a^3_{j_1},\ldots,a^1_{j_m}, a^2_{j_m}, a^3_{j_m}, c^1, c^2, e^3]$, i.e. $\mathcal{R}\models\\\exists x^1_{m+1}\exists x^2_{m+1}\exists x^3_{m+1}(\widehat{\varphi'}\wedge\kappa_{m+1})[a^1_{j_1}, a^2_{j_1}, a^3_{j_1},\ldots,a^1_{j_m}, a^2_{j_m}, a^3_{j_m}]$.\\
	\textbf{Case 2}: $d^1\in\{b^1_{j_1},b^2_{j_1},\ldots,b^1_{j_m},b^2_{j_m}\}$\\
	We proved that $b^1_{j_1}$, $b^2_{j_1}$, $b^3_{j_1},\ldots,b^1_{j_m}$, $b^2_{j_m}$, $b^3_{j_m}$ are corresponding to $a^1_{j_1}$, $a^2_{j_1}$, $a^3_{j_1},\ldots,a^1_{j_m}$, $a^2_{j_m}$, $a^3_{j_m}$ and hence $b^1_{j_k}\cdot b^2_{j_k}=-1$ for every $k\in\{1,\ldots,m\}$. We have also that $d^1=b^1_{j_k}$ or $d^1=b^2_{j_k}$ for some $k\in\{1,\ldots,m\}$ and $\mathcal{R}\models\kappa_{m+1}[b^1_{j_1},b^2_{j_1},\ldots,b^1_{j_m},b^2_{j_m},d^1,d^2]$. Consequently $d^1\cdot d^2=-1$.\\
	\textbf{Case 2.1}: $d^1=b^1_{j_k}$ for some $k\in\{1,\ldots,m\}$\\
	We denote $c^1=a^1_{j_k}$, $c^2=a^2_{j_k}$. We will prove that $b^1_{j_1}$, $b^2_{j_1}$, $b^3_{j_1},\ldots,b^1_{j_m}$, $b^2_{j_m}$, $b^3_{j_m}$, $d^1$, $d^2$, $d^3$ are corresponding to $a^1_{j_1}$, $a^2_{j_1}$, $a^3_{j_1},\ldots,a^1_{j_m}$, $a^2_{j_m}$, $a^3_{j_m}$, $c^1$, $c^2$, $d^3$. We proved that $\mathcal{R}\models\kappa_l[a^1_{j_1},a^2_{j_1},\ldots,a^1_{j_l},a^2_{j_l}]$ for every $l\in\{1,\ldots,m\}$. We will prove that $\mathcal{R}\models\kappa_{m+1}[a^1_{j_1},a^2_{j_1},\ldots,a^1_{j_m},a^2_{j_m},c^1,c^2]$. Since $\mathcal{R}\models\kappa_k[a^1_{j_1},a^2_{j_1},\ldots,a^1_{j_k},a^2_{j_k}]$, $a^1_{j_k}\neq a^2_{j_k}$, i.e. $c^1\neq c^2$. It remains to prove that for every $l\in\{1,\ldots,m\}$, ($a^1_{j_l}=a^1_{j_k}$ iff $a^2_{j_l}=a^2_{j_k}$) and ($a^1_{j_l}=a^2_{j_k}$ iff $a^2_{j_l}=a^1_{j_k}$). If $l=k$, then these equivalences are obvious. If $l<k$, then they follow from $\mathcal{R}\models\kappa_k[a^1_{j_1}, a^2_{j_1},\ldots,a^1_{j_k}, a^2_{j_k}]$; if $l>k$, then they follow from $\mathcal{R}\models\kappa_l[a^1_{j_1}, a^2_{j_1},\ldots,a^1_{j_l}, a^2_{j_l}]$. Thus $\mathcal{R}\models\kappa_{m+1}[a^1_{j_1}, a^2_{j_1},\ldots,a^1_{j_m}, a^2_{j_m},c^1,c^2]$. Conditions 2), 3) and 4) of Definition~\ref{def2} are fulfilled for every step from $1$ to $m$ because $b^1_{j_1}$, $b^2_{j_1}$, $b^3_{j_1},\ldots,b^1_{j_m}$, $b^2_{j_m}$, $b^3_{j_m}$ are corresponding to $a^1_{j_1}$, $a^2_{j_1}$, $a^3_{j_1},\ldots,a^1_{j_m}$, $a^2_{j_m}$, $a^3_{j_m}$. Obviously at step $m+1$ condition 2) of Definition~\ref{def2} is fulfilled.
	
	We will prove that condition 3) of Definition~\ref{def2} is fulfilled at step $m+1$. Let $c^1=a^1_{j_l}$ for some $l\in\{1,\ldots,m\}$. We will prove that $d^1=b^1_{j_l}$. Using that $c^1=a^1_{j_k}=a^1_{j_l}$ and the fact that $b^1_{j_1}$, $b^2_{j_1}$, $b^3_{j_1},\ldots,b^1_{j_m}$, $b^2_{j_m}$, $b^3_{j_m}$ are corresponding to $a^1_{j_1}$, $a^2_{j_1}$, $a^3_{j_1},\ldots,a^1_{j_m}$, $a^2_{j_m}$, $a^3_{j_m}$, we obtain that $b^1_{j_k}=b^1_{j_l}$. Consequently $d^1=b^1_{j_l}$.
	
	We will prove that condition 4) of Definition~\ref{def2} is fulfilled at step $m+1$. Let $c^1=a^2_{j_l}$ for some $l\in\{1,\ldots,m\}$.\\
	\textbf{Case a}: $l>k$\\
	Since $\mathcal{R}\models\kappa_l[a^1_{j_1},a^2_{j_1},\ldots,a^1_{j_l},a^2_{j_l}]$, $a^1_{j_k}=a^2_{j_l}$ iff $a^2_{j_k}=a^1_{j_l}$ and hence $a^1_{j_l}=a^2_{j_k}$. We have that $b^1_{j_1}$, $b^2_{j_1}$, $b^3_{j_1},\ldots,b^1_{j_m}$, $b^2_{j_m}$, $b^3_{j_m}$ are corresponding to $a^1_{j_1}$, $a^2_{j_1}$, $a^3_{j_1},\ldots,a^1_{j_m}$, $a^2_{j_m}$, $a^3_{j_m}$ and by condition 4) of Definition~\ref{def2} we get $b^1_{j_l}=b^2_{j_k}$. Consequently $b^1_{j_k}=b^2_{j_l}$ and hence $d^1=b^2_{j_l}$.\\
	\textbf{Case b}: $k\geq l$\\
	We have proved that $a^1_{j_k}\neq a^2_{j_k}$, so $k\neq l$. Thus $k>l$. We have that $a^1_{j_k}=a^2_{j_l}$ and $b^1_{j_1}$, $b^2_{j_1}$, $b^3_{j_1},\ldots,b^1_{j_m}$, $b^2_{j_m}$, $b^3_{j_m}$ are corresponding to $a^1_{j_1}$, $a^2_{j_1}$, $a^3_{j_1},\ldots,a^1_{j_m}$, $a^2_{j_m}$, $a^3_{j_m}$. Consequently $b^1_{j_k}=b^2_{j_l}$ and hence $d^1=b^2_{j_l}$.
	
	Consequently condition 4) of Definition~\ref{def2} is fulfilled at step $m+1$. Clearly for every step, conditions 5) and 6) of Definition~\ref{def2} are fulfilled. Thus we proved that $b^1_{j_1}$, $b^2_{j_1}$, $b^3_{j_1},\ldots,b^1_{j_m}$, $b^2_{j_m}$, $b^3_{j_m}$, $d^1$, $d^2$, $d^3$ are corresponding to $a^1_{j_1}$, $a^2_{j_1}$, $a^3_{j_1},\ldots,a^1_{j_m}$, $a^2_{j_m}$, $a^3_{j_m}$, $c^1$, $c^2$, $d^3$. By the induction hypothesis, $\mathcal{R}\models\widehat{\varphi'}[a^1_{j_1},a^2_{j_1},a^3_{j_1},\ldots,a^1_{j_m},a^2_{j_m},a^3_{j_m},c^1,c^2,d^3]$ iff $\mathcal{R}\models\widehat{\varphi'}[b^1_{j_1},b^2_{j_1},b^3_{j_1},\ldots,b^1_{j_m},b^2_{j_m},b^3_{j_m},\\d^1,d^2,d^3]$. Consequently $\mathcal{R}\models\widehat{\varphi'}[a^1_{j_1},a^2_{j_1},a^3_{j_1},\ldots,a^1_{j_m},a^2_{j_m},a^3_{j_m},c^1,c^2,d^3]$. We have also that $\mathcal{R}\models\kappa_{m+1}[a^1_{j_1},a^2_{j_1},\ldots,a^1_{j_m},a^2_{j_m},c^1,c^2]$. Consequently $\mathcal{R}\models\exists x^1_{m+1}\exists x^2_{m+1}\exists x^3_{m+1}(\widehat{\varphi'}\wedge\kappa_{m+1})[a^1_{j_1},a^2_{j_1},a^3_{j_1},\ldots,a^1_{j_m},a^2_{j_m},a^3_{j_m}]$.\\
	\textbf{Case 2.2}: $d^1=b^2_{j_k}$ for some $k\in\{1,\ldots,m\}$\\
	We denote $c^1=a^2_{j_k}$, $c^2=a^1_{j_k}$. In a similar way as in Case 2.1 we prove that $b^1_{j_1}$, $b^2_{j_1}$, $b^3_{j_1},\ldots,b^1_{j_m}$, $b^2_{j_m}$, $b^3_{j_m}$, $d^1$, $d^2$, $d^3$ are corresponding to $a^1_{j_1}$, $a^2_{j_1}$, $a^3_{j_1},\ldots,a^1_{j_m}$, $a^2_{j_m}$, $a^3_{j_m}$, $c^1$, $c^2$, $d^3$ and $\mathcal{R}\models\exists x^1_{m+1}\exists x^2_{m+1}\exists x^3_{m+1}(\widehat{\varphi'}\wedge\kappa_{m+1})[a^1_{j_1},a^2_{j_1},a^3_{j_1},\ldots,a^1_{j_m},\\a^2_{j_m},a^3_{j_m}]$. $\Box$
	\end{proof}
	
	\begin{lemma}\label{lemma8}
		Let $\varphi$ be a formula in $\mathcal{L}$ with free variables $x_1,\ldots,x_n$ and $a_1,\ldots,a_n\\\in A^{\ast}$. Then $\mathcal{A}^{\ast}\models\varphi[a_1,\ldots,a_n]$ iff $\mathcal{R}\models\widehat{\varphi}[a^1_1,a^2_1,a^3_1,\ldots,a^1_n,a^2_n,a^3_n]$.
	\end{lemma}
	\begin{proof}
		Induction on $\varphi$. The base of induction, the cases when $\varphi$ is $\neg\varphi'$ or $\varphi$ is $\varphi'\wedge\varphi''$ are trivial. Let $\varphi$ be $\exists x_{n+1}\varphi'$ and $\varphi'$ has free variables $x_1,\ldots,x_{n+1}$. Let $a_1,\ldots,a_n\in A^{\ast}$. We will prove that $\mathcal{A}^{\ast}\models\exists x_{n+1}\varphi'[a_1,\ldots,a_n]$ iff $\mathcal{R}\models\exists x^1_{n+1}\exists x^2_{n+1}\exists x^3_{n+1}(\widehat{\varphi'}\wedge\kappa_{n+1})[a^1_1,a^2_1,a^3_1,\ldots,a^1_n,a^2_n,a^3_n]$.\\
		$\bm{\Rightarrow)}$ The proof is trivial.\\
		$\bm{\Leftarrow)}$ There are $a^1_{n+1}$, $a^2_{n+1}$, $a^3_{n+1}\in\mathbb{R}\setminus\{0\}$ such that $\mathcal{R}\models(\widehat{\varphi'}\wedge\kappa_{n+1})[a^1_1,a^2_1,a^3_1,\ldots,\\a^1_n,a^2_n,a^3_n,a^1_{n+1},a^2_{n+1},a^3_{n+1}]$. If $a^1_{n+1}\notin\{a^1_1,a^2_1,\ldots,a^1_n,a^2_n\}$, let $b^1_{n+1}$ be a real number, different from $0$, not belonging to $\{a^1_1,a^2_1,\ldots,a^1_n,a^2_n\}$; if $a^1_{n+1}=a^1_k$ for some $k\in\{1,\ldots,n\}$, we denote $b^1_{n+1}=a^1_k$; if $a^1_{n+1}=a^2_k$ for some $k\in\{1,\ldots,n\}$, we denote $b^1_{n+1}=a^2_k$. We denote $b^2_{n+1}=-\frac{1}{b^1_{n+1}}$. Clearly the definition is correct. It can be easily verified that $a^1_1$, $a^2_1$, $a^3_1,\ldots,a^1_n$, $a^2_n$, $a^3_n$, $b^1_{n+1}$, $b^2_{n+1}$, $a^3_{n+1}$ are corresponding to $a^1_1$, $a^2_1$, $a^3_1,\ldots,a^1_n$, $a^2_n$, $a^3_n$, $a^1_{n+1}$, $a^2_{n+1}$, $a^3_{n+1}$. By Lemma~\ref{lemma7}, $\mathcal{R}\models\widehat{\varphi'}[a^1_1,a^2_1,a^3_1,\ldots,a^1_n,a^2_n,a^3_n,a^1_{n+1},a^2_{n+1},a^3_{n+1}]$ iff $\mathcal{R}\models\widehat{\varphi'}[a^1_1,a^2_1,a^3_1,\ldots,a^1_n,a^2_n,a^3_n,b^1_{n+1},b^2_{n+1},a^3_{n+1}]$. Consequently $\mathcal{R}\models\widehat{\varphi'}[a^1_1,a^2_1,a^3_1,\\\ldots,a^1_n,a^2_n,a^3_n,b^1_{n+1},b^2_{n+1},a^3_{n+1}]$. Let $a_{n+1}$ be the line with equation $y=b^1_{n+1}x+a^3_{n+1}$. Obviously $a_{n+1}\in A^{\ast}$. The coordinates of $a_{n+1}$ are $b^1_{n+1}$, $b^2_{n+1}$, $a^3_{n+1}$. By the induction hypothesis, $\mathcal{A}^{\ast}\models\varphi'[a_1,\ldots,a_n,a_{n+1}]$ iff $\mathcal{R}\models\widehat{\varphi'}[a^1_1,a^2_1,a^3_1,\ldots,a^1_n,\\a^2_n,a^3_n,b^1_{n+1},b^2_{n+1},a^3_{n+1}]$. Consequently $\mathcal{A}^{\ast}\models\varphi'[a_1,\ldots,a_n,a_{n+1}]$ and hence $\mathcal{A}^{\ast}\models\exists x_{n+1}\varphi'[a_1,\ldots,a_n]$. $\Box$
	\end{proof}
	
	\bigskip
	It can be easily verified that $\widehat{\varphi}$ can be obtained algorithmically from $\varphi$ by using of memory, polynomial in the size of $\varphi$.
	
	By Lemma~\ref{lemma3}, Lemma~\ref{lemma8} and Lemma~\ref{lemma4}, $SAPP\models\varphi$ iff $\mathcal{A}^{\ast}\models\varphi$ iff $\mathcal{R}\models\widehat{\varphi}$ iff $\widehat{\varphi}\in EQ^{\infty}$.
	
	Now we will prove that the problem if a closed formula in $\mathcal{L}$ logically follows from $SAPP$ is PSPACE-hard. We will prove that for every closed formula in $\mathcal{L}_1=\langle\ ;\ ;=\rangle$ $\varphi_1$, $\varphi_1\in EQ^{\infty}$ iff $SAPP\models\varphi_1$. The formula $\varphi_1$ is a formula in $\mathcal{L}$ too. By Lemma~\ref{lemma3}, $SAPP\models\varphi_1$ iff $\mathcal{A}^{\ast}\models\varphi_1$. Similarly of Lemma~\ref{lemma4} we prove that $\mathcal{A}^{\ast}\models\varphi_1$ iff $\varphi_1\in EQ^{\infty}$. Thus we proved that for every closed formula in $\mathcal{L}_1$ $\varphi_1$, $\varphi_1\in EQ^{\infty}$ iff $SAPP\models\varphi_1$. We have also that the membership problem in $EQ^{\infty}$ is PSPACE-hard. Consequently the problem if a closed formula in $\mathcal{L}$ logically follows from $SAPP$ is PSPACE-hard. $\Box$
\end{proof}	

\section{Some additional results}

\begin{proposition}
	The theory $SAPP$ is not finitely axiomatizable. 
\end{proposition}
\begin{proof}
	For the sake of contradiction suppose that $SAPP$ is axiomatized by a finite set of formulas $\Gamma$. Let the maximal quantifier rank of a formula of $\Gamma$ be $n$. We denote $k=n+1$. Obviously $k>0$. Let $S$ be the set of the lines with equation of the kind $y=ax+b$, where $a\in\{1,\ldots,k,-1,-\frac{1}{2},\ldots,-\frac{1}{k}\}$, $b\in\{1,\ldots,k\}$. Obviously $S\subseteq F^2_{\mathbb{Q}}$. Let $\mathcal{S}$ be the substructure of $\mathcal{F}^2_{\mathbb{Q}}$ with universe $S$. We consider the structures $\mathcal{F}^2_{\mathbb{Q}}$ and $\mathcal{S}$ and Ehrenfeucht-Fra\"{\i}ss\`{e}'s game with length $k$. It can be easily found a winning strategy for the second player. Consequently $\mathcal{F}^2_{\mathbb{Q}}$ and $\mathcal{S}$ are $k$ elementary equivalent. Consequently for every formula $\varphi$ of $\Gamma$, $\mathcal{F}^2_{\mathbb{Q}}\models\varphi$ iff $\mathcal{S}\models\varphi$. Consequently $\mathcal{S}\models\Gamma$ and hence $\mathcal{S}\models SAPP$ but $\mathcal{S}\not\models\lambda_{1k}$ - a contradiction. $\Box$
\end{proof}

\bigskip
It can be easily proved the following proposition:

\begin{proposition}
	The theory $SAPP$ is $\omega$-categorical. 
\end{proposition}

\begin{proposition}
	The theory $SAPP$ is not $\alpha$-categorical for every uncountable cardinality $\alpha$. 
\end{proposition}
\begin{proof}
	We denote by $\mathbb{Z}^{\ast}$ the set of all non-zero integers. Let $\alpha$ be an arbitrary uncountable cardinal. Let $A=\alpha\times\mathbb{Z}^{\ast}$; $B=\mathbb{Z}^{\ast}\times\alpha$. We consider the structure $\mathcal{A}=(A,O,=)$, where $O((x_1,x_2),(y_1,y_2))\stackrel{def}{\Longleftrightarrow}x_1=y_1$ and $x_2\cdot y_2<0$. We consider also the structure $\mathcal{B}=(B,O,=)$, where $O((x_1,x_2),(y_1,y_2))\stackrel{def}{\Longleftrightarrow}\frac{x_1}{y_1}=-1$. 
	
	Clearly $A$ and $B$ have cardinality $\alpha$.
	
	It can be easily proved the following lemma:
	
	\begin{lemma}\label{lemma9}
		\textbf{(i)} Every equivalence class of $A$ is countable and has the form $\{(\gamma,1),(\gamma,2),\ldots,(\gamma,n),\ldots\}$ or $\{(\gamma,-1),(\gamma,-2),\ldots,(\gamma,-n),\ldots\}$, where $\gamma<\alpha$;\\
		\textbf{(ii)} Every equivalence class of $B$ is uncountable and has the form $\{(s,\beta):\ \beta<\alpha\}$, where $s\in\mathbb{Z}^{\ast}$.
	\end{lemma}
	
	Using Lemma~\ref{lemma9}, it can be easily verified that $\mathcal{A}\models SAPP$ and $\mathcal{B}\models SAPP$. 
	
	Suppose for the sake of contradiction that there is an isomorphism $f:\ \mathcal{A}\rightarrow\mathcal{B}$. We will prove the following lemma:
	
	\begin{lemma}\label{lemma10}
		For any $x$, $y\in A$, $[x]=[y]$ iff $[f(x)]=[f(y)]$.
	\end{lemma}
	\begin{proof}
		$\bm{\Rightarrow)}$ Obviously $\neg O(f(x),f(y))$. There is $z\in B$ such that $O(f(y),z)$. Since $f$ is an isomorphism, $z=f(z_1)$ for some $z_1\in A$. Obviously $O(y,z_1)$ and hence $O(x,z_1)$; so $O(f(x),z)$; so $[z]R_2[f(x)]$ but we also have $[z]R_2[f(y)]$, so $[f(x)]=[f(y)]$.\\
		$\bm{\Leftarrow)}$ The proof is symmetric. $\Box$
	\end{proof}
	
	Let $\gamma<\alpha$. We have $(\gamma,1)\in A$. By Lemma~\ref{lemma9}, $[(\gamma,1)]$ is countable and has the form $\{(\gamma,1),(\gamma,2),\ldots,(\gamma,n),\ldots\}$. We have $f((\gamma,1))=(s,\delta)$, where $s\in\mathbb{Z}^{\ast}$, $\delta<\alpha$. By Lemma~\ref{lemma9}, $[(s,\delta)]=\{(s,\beta):\ \beta<\alpha\}$. Using Lemma~\ref{lemma10}, we get that for every positive integer $i$, $f((\gamma,i))=(s,\beta_i)$, where $\beta_i<\alpha$ and $\beta_i\neq\beta_j$ for $i\neq j$. Since $\alpha$ is uncountable cardinal, there is $\eta<\alpha$ such that $\eta\neq\beta_i$ for every $i$. We have $(s,\eta)=f(t)$ for some $t\in A$ and $[(s,\eta)]=[(s,\delta)]$. By Lemma~\ref{lemma10}, $[(s,\eta)]=[(s,\delta)]$ iff $[t]=[(\gamma,1)]$; so $[t]=[(\gamma,1)]$; so $t=(\gamma,i)$ for some positive integer $i$. We have $f((\gamma,i))=(s,\beta_i)=f(t)=(s,\eta)$ - a contradiction. Consequently there is no isomorphism $f:\ \mathcal{A}\rightarrow\mathcal{B}$.
	
	Thus $SAPP$ is not $\alpha$-categorical. $\Box$
\end{proof}

\end{document}